\documentclass[12pt]{amsart}
\usepackage{amssymb, latexsym, mathrsfs,   color, verbatim}
\usepackage[pdfstartview=FitV,linkcolor=blue,citecolor=blue,urlcolor=blue]{hyperref}

\newtheorem {theorem}    {Theorem}[section]

\newtheorem {lemma}      [theorem]    {Lemma}
\newtheorem {corollary}  [theorem]    {Corollary}
\newtheorem {proposition}[theorem]    {Proposition}

\newtheorem {property} [theorem] {Property}

\theoremstyle{definition}
\newtheorem{definition}[theorem]{Definition}

\newtheorem{ex}[theorem]{Example}

\def\la{\langle}
\def\ra{\rangle}

\def\Z{\mathbb{Z}}
\def\R{\mathbb{R}}
\def\C{\mathbb{C}}

\def\a{\alpha}

\numberwithin{equation}{section}

\begin{document}

\title[On the convergence of  Kac--Moody Eisenstein series]{On the convergence  of  Kac--Moody Eisenstein series}

\date{\today}

\author[L. Carbone]{Lisa Carbone$^{\dagger}$}
\address{Department of Mathematics, Rutgers University, Piscataway, NJ 08854-8019, U.S.A.}
\email{carbonel@math.rutgers.edu}
\thanks{$^{\dagger}$This work was supported in part by NSF grant DMS--1101282.}

\author[H. Garland]{Howard Garland}
\address{Department of
Mathematics, Yale University, New Haven, CT 06511, U.S.A.}
\email{garland-howard@yale.edu}

\author[K.-H. Lee]{Kyu-Hwan Lee$^{\star}$}
\thanks{$^{\star}$This work was partially supported by a grant from the Simons Foundation (\#318706).}
\address{Department of
Mathematics, University of Connecticut, Storrs, CT 06269, U.S.A.}
\email{khlee@math.uconn.edu}

\author[D. Liu]{Dongwen Liu$^{\diamond}$}
\thanks{$^{\diamond}$This work was partially supported by  the Fundamental Research Funds for the Central Universities 2016QNA2002.}
\address{School of Mathematical Sciences, Zhejiang University, Hangzhou 310058, P.R. China}\email{maliu@zju.edu.cn}

\author[S. D. Miller]{Stephen D. Miller$^\square$}
\thanks{$^\square$This work was supported by National Science Foundation grant DMS-1801417.}
\address{Department of Mathematics, Rutgers University, Piscataway, NJ 08854-8019, U.S.A.}
\email{miller@math.rutgers.edu}


\begin{abstract}  Let $G$ be a representation-theoretic Kac--Moody  group associated to a nonsingular symmetrizable generalized Cartan matrix. We first consider Kac-Moody analogs of Borel Eisenstein series (induced from quasicharacters  on the Borel), and prove they converge almost everywhere inside the Tits cone for arbitrary spectral parameters in the Godement range.  We then use this result to show the full absolute convergence everywhere inside the Tits cone (again for spectral parameters in the Godement range) for a  class of Kac--Moody groups satisfying a certain combinatorial property, in particular for rank-2 hyperbolic groups.
\end{abstract}

\maketitle

\section{Introduction}\label{sec:intro}

The fundamental theory of Eisenstein series on reductive groups   has played pivotal roles in the formulation of   Langlands' functoriality conjecture \cite{La1, LA1}, and in the study of $L$--functions by means of the Langlands--Shahidi method (e.g., \cite{KimSh, Kim}).

The extension of the theory of Eisenstein series to Kac--Moody groups is of great interest, due to its conjectural roles in some of the central problems in number theory, such as establishing important analytic properties of $L$-functions \cite{BFH,Sh}.   Recently,
Eisenstein series on exceptional Lie groups have been shown to occur explicitly as coefficients of correction terms in certain maximally
supersymmetric string theories   \cite{GRV,GMRV,GMV}. Conjectural Eisenstein series on Kac--Moody groups appear in recent developments  in string theory \cite{FK,FKP}.  All of these potential applications require a proof of convergence (and in some cases, analytic continuation) of Kac-Moody Eisenstein series.

In his papers \cite{G04, G06, GMS1, GMS2, GMS3, GMS4}, Garland extended the classical theory of Eisenstein series to arithmetic  quotients  $G_\Z\backslash G_\R/ K$ of affine Kac--Moody groups $G$.
 In particular, he established   absolute convergence for spectral parameters in a Godement range, and then proved a meromorphic continuation beyond it.  Absolute convergence  has been generalized to affine Kac--Moody groups over number fields by Liu \cite{Li}.  Garland, Miller, and Patnaik \cite{GMP}  showed that affine Eisenstein series  induced from cusp forms over $\mathbb Q$ are entire functions of the spectral parameter. It should be mentioned that earlier results  were established in the function field setting by \cite{BK}; related results also appear in \cite{Ka,LL,P}.

Beyond the affine case,  Carbone, Lee, and Liu \cite{CLL} studied Eisenstein series on the rank $2$ hyperbolic Kac--Moody groups with symmetric generalized Cartan matrices, and established almost-everywhere convergence of the series. However, they could not obtain everywhere convergence as the method in the affine case does not generalize directly to the hyperbolic case.

\noindent{\bf{Borel Eisenstein series}}

In this paper we consider the problem of establishing the absolute convergence of Eisenstein series on arbitrary Kac-Moody groups, for spectral parameters $\lambda$ in the traditionally-studied Godement range (i.e., $\operatorname{Re}(\lambda-\rho)$ is strictly dominant).  We indeed establish their almost-everywhere convergence in this generality, and additionally show the absolute convergence for a wide class of groups.

More precisely, let $G$ be a representation-theoretic  Kac--Moody group, and let $\mathfrak g$ be the corresponding real Kac--Moody algebra with a fixed Cartan subalgebra $\mathfrak h$.  We assume that $\frak g$ is infinite-dimensional and non-affine, since the finite-dimensional and affine cases have been well-studied.  Let $r=\dim(\mathfrak h)$ denote the rank of $G$, $I$ the index set $\{1,\ldots,r\}$, and $\Phi_+$ (resp., $\Phi_{-}$) the positive (resp., negative) roots of $\frak g_\C$.  Then $G_{\mathbb R}$ has
the Iwasawa decomposition $G_{\mathbb R}=UA^+K$, where $U$ is a maximal pro-unipotent subgroup, $A^+$ is the connected component of a maximal torus, and $K$ is a subgroup of $G$ playing the role  of the maximal compact subgroup from the finite-dimensional theory.
(See Section~\ref{Zform} for more details.)

 We formally define the Borel Kac--Moody Eisenstein series $E_\lambda(g)$ for  $g \in G_{\mathbb R}$  and $\lambda \in\mathfrak{h}^*_\mathbb{C}$  by
\begin{equation}\label{formaldefofEis}
E_\lambda (g) \ \  = \ \ \sum_{\gamma \in (\Gamma \cap B) \backslash \Gamma} a(\gamma g)^{\lambda+\rho}\,,
\end{equation}
where $a(g)$ is the $A^+$--component of the Iwasawa decomposition of $g$, $\rho$ is the Weyl vector, $\Gamma=G_\Z$ is the arithmetic subgroup defined at the end of Section~\ref{Zform}, and $B\supset NA^+$ is a Borel subgroup.  The goal of this paper is to study the convergence and analyticity of this formal sum.
Our main result can be stated as follows:

\begin{theorem} \label{thm-intro}
Assume  that $\lambda\in\mathfrak{h}^*_\mathbb{C}$ satisfies $\operatorname{Re}(\langle\lambda, \alpha_i^\vee\rangle)>1$ for each simple coroot $\alpha_i^\vee$, $i \in I$, and that Property~\ref{conj} holds. Then the Kac--Moody Eisenstein series $E_\lambda(g)$ converges absolutely for $g\in \Gamma UA_\mathfrak{C}K$, where $A_\mathfrak{C}\subset A^+$ is the image of the Tits cone $\mathfrak C$ under the exponential map $\exp: \mathfrak{h} \to A^+$.
\end{theorem}

\noindent
The condition on $\lambda$ is precisely the Godement range, and  appears in the classical theory \cite{LA1}. Property~\ref{conj} and hence
the conclusions of Theorem~\ref{thm-intro} hold for all $G$ of rank $2$ (Proposition \ref{prop-rank2}), as well when  the Cartan matrix is symmetric and has  sufficiently large entries (Proposition \ref{prop-generic}).
 However,
 Property~\ref{conj} is not true for all   Kac--Moody root systems -- it even fails in the  finite-dimensional   Example \ref{counterex}.

Though a Godement condition is very natural on the spectral parameter,
it is not clear what the full range of absolute convergence of $E_\lambda(g)$ is in the variable $g$.
  Theorem~\ref{thm-intro} adopts a Tits cone constraint, one which naturally occurs in the literature  when studying simpler sums, e.g., in the Lemma of  Looijenga  \cite{Lo} (see also \cite[\S10.6]{K})   quoted in  Lemma~\ref{lo}, which corresponds to $g$ in the set $UA_{\mathfrak C}K$.  The extension to the domain $\Gamma UA_\mathfrak{C}K$ -- which Proposition \ref{nonstable} shows is in fact larger --  comes from the $\Gamma$-invariance in (\ref{formaldefofEis}).

In order to prove Theorem~\ref{thm-intro}  we consider the constant term $E^\sharp_\lambda(g)$ of the series  $E_\lambda(g)$, which is computed by the Gindikin--Karpelevich formula, and establish its absolute convergence in Theorem~\ref{3.3}. This is achieved by  using Looijenga's Lemma and  the observation that $M^{\ell (w)} c(\lambda, w)$ is bounded  for any fixed  $M>0$, as $w$ varies over  the Weyl group $W$. Here the function $c(\lambda, w)$ is the product of ratios of the Riemann $\zeta$-function defined in (\ref{eee}).  Unlike Theorem~\ref{thm-intro}, which assumes Property~\ref{conj}, Theorem~\ref{3.3} and its almost-everywhere convergence Corollary~\ref{cor} hold for {\it all} symmetrizable $G$.

In Section \ref{Conv} we show that the Kac--Moody Eisenstein series $E_\lambda(g)$ is dominated by a sum over the Weyl group $W$, which is a close variant   of the constant term $E^\sharp_\lambda(g)$.  Having proved the convergence of $E^\sharp_\lambda(g)$ in Theorem \ref{3.3}, we use some consequences of Property~\ref{conj} to establish  Theorem \ref{thm-intro}.

\noindent {\bf Remark:} Though Theorem~\ref{thm-intro} is stated only for Borel Eisenstein series, its conclusions (with modifications just as in the classical finite-dimensional setting) hold for {\em cuspidally} induced Eisenstein series as well, at least for parabolics with finite-dimensional Levi components.  This is because cusp forms on finite-dimensional semisimple groups are bounded, hence bounded above by certain Borel Eisenstein series on the Levi component.  A parabolic Eisenstein series for $G$ induced from such an Eisenstein series is itself a Borel Eisenstein series for $G$, hence Theorem~\ref{thm-intro} implies a corresponding convergence statement for Eisenstein series induced from cusp forms.  See   \cite[Proposition 12.6]{Bo},
\cite[Proposition II.1.5]{MW}, and \cite[\S3]{G11}, where this argument is carried out in more detail in simpler settings.  For the same reason, the weaker statements Theorem~\ref{3.3} and Corollary~\ref{cor} also transfer to results for almost-everywhere convergence for cuspidally-induced Eisenstein series.

We expect Kac--Moody Eisenstein series to provide   interesting applications.  In analogy with $SL_2(\mathbb{Z})$, the group $E_{10}(\mathbb{Z})$ is conjectured to be the discrete invariance group for certain functions that arise in   11-dimensional supersymmetric string theory \cite{DKN,Ga}. Automorphic forms on $E_{10}$ and $E_{11}$  are conjectured  to encode  higher derivative corrections in string theory and M--theory \cite{DN,DHHKN,FGKP,W1}.

Eisenstein series  on Kac--Moody groups $G_\R$ are invariant under translations by $G_\Z$ and hence have Fourier expansions.  In \cite{CLL}, the authors defined and calculated the degenerate Fourier coefficients for Eisenstein series on rank $2$ hyperbolic Kac--Moody  groups over $\R$. Fleig \cite{Fl}  gave the Fourier integrals needed to obtain the constant term and  higher order Fourier
modes for Eisenstein series on $E_9$, $E_{10}$, and $E_{11}$, and showed how the  string scattering amplitude ``collapse mechanism'' of \cite{FK} extends to the higher order Fourier modes. He also gave explicit expressions for the constant terms and Fourier modes of some Kac--Moody Eisenstein series.
 However, these calculations  tacitly assume absolute convergence (and in some cases,  meromorphic continuation) of the Kac--Moody Eisenstein series, which has yet to be accomplished for $E_{10}$ and $E_{11}$; Theorem~\ref{thm-intro} does not  cover these cases because of its restrictive assumption of Property~\ref{conj}.

It is a pleasure to thank our colleagues Alexander Braverman, Daniel Bump, Solomon Friedberg, Michael B. Green, Henrik Gustafsson, Jeffrey Hoffstein, David Kazhdan, Henry Kim, Axel Kleinschmidt, Manish Patnaik, Daniel Persson, Peter Sarnak,  Freydoon Shahidi, and Pierre Vanhove for their helpful conversations.

\section{Kac--Moody groups} \label{Zform}

Let $I=\{ 1, 2, \dots , r \}$,
${\sf A}=(a_{ij})_{i,j \in I}$ be an $r \times r$  symmetrizable  generalized Cartan matrix, and $(\mathfrak h, \Delta, \Delta^\vee)$ be a realization of ${\sf A}$, where $\Delta=\{ \alpha_1, ... , \alpha_r \} \subset \mathfrak h^*$ and  $\Delta^\vee=\{ \alpha^\vee_1, ... , \alpha^\vee_r \} \subset \mathfrak h$ are the set of simple roots and set of simple coroots, respectively (see \cite[\S1]{K} for definitions).

Throughout this paper we shall make the simplifying assumption that ${\sf A}$ is nonsingular, which means that $\mathfrak{h}$ and $\mathfrak{h}^*$  are spanned by the simple roots $\alpha_i$ and simple coroots $\alpha_i^\vee$, respectively. In particular we intentionally exclude the affine case, which  has been studied extensively in Garland's works and in \cite{GMP} (and which has a somewhat different flavor anyhow).

Recall that  $\langle\alpha_j, \alpha^\vee_i\rangle = a_{ij}$ for $i,j \in I$, where $\langle \cdot,\cdot\rangle$ is the natural pairing between $\frak{h}^*$ and $\frak{h}$.
Denote the fundamental weights by $\varpi_i\in\mathfrak{h}^*$, $i\in I$, which form the   basis of $\frak h^*$ dual to the $\alpha_i^\vee$.  Their integral span is the weight lattice $P$.

Let $\mathfrak g_{\C}=\mathfrak g_{\C}({\sf A})$ be the Kac--Moody algebra associated to $(\mathfrak h, \Delta, \Delta^\vee)$.  We denote by $\Phi$ the set of roots of $\mathfrak g_{\C}$ and have $\Phi = \Phi_+ \sqcup \Phi_{-}$, where $\Phi_+$ (resp. $\Phi_{-}$) is the set of positive (resp. negative) roots corresponding to the choice of $\Delta$. 
 Let $w_i:=w_{\a_i}$ denote the simple Weyl reflection associated to the simple root $\alpha_i$; the $w_i$ for $i\in I$  generate the Weyl group $W$ of $\frak g_\C$.  A root $\alpha\in\Phi$ is called a {\it real root} if there exists $w\in W$ such that $w\alpha$ is a
simple root. A root $\alpha$ which is not real is called {\it
imaginary}.  For each real root $\alpha$ written as $w\alpha_i$ for some $w\in W$ and $i\in I$, its associated coroot  is well-defined by the formula $\alpha^\vee= w\alpha_i^\vee$ \cite[1.3.8]{Kumar}.

For $i\in I$ let $e_i$ and $f_i$ be the Chevalley generators of $\mathfrak{g}_\C$; we denote by $\frak g$ the real Lie subalgebra they generate, so that $\frak g_\C=\frak g\otimes_\R \C$. Let ${\mathcal U}_{\mathbb{C}}$ be the universal enveloping algebra of $\mathfrak{g}_\C$.
Let $\Lambda\subseteq \mathfrak{h}^*$ be the integral linear span of  the simple roots $\alpha_i$,   $i\in I$, and let $\Lambda^{\vee}\subseteq
\mathfrak{h}$ be the integral  span of the simple coroots
$\alpha^{\vee}_i$,   $i\in I$. Let  ${\mathcal U}_{{\mathbb{Z}}}\subseteq {\mathcal U}_{{\mathbb{C}}}$ be the ${\mathbb{Z}}$--subalgebra generated by
$$\dfrac{e_i^{m}}{m!},\ \dfrac{f_i^{m}}{m!},\ \text{and} \ \left (\begin{matrix}
h\\ m\end{matrix}\right ) = \frac{h(h-1)\cdots(h-m+1)}{m!}$$ for $i\in I$,
$h\in\Lambda^{\vee}$, and
$m\geq 0$ (see \cite[\S4]{CG}).

We will now review the representation theoretic Kac--Moody groups  $G_{F}$ associated to $\frak g$ and a field $F$.  The paper \cite{CG} provides an explicit construction for arbitrary fields $F$, and we shall presently  review the construction for the fields $\mathbb{Q}$, $\R$, and $\mathbb{Q}_p$.  Keeping this in mind, we assume $F\supset \mathbb{Q}$.

Let $(\pi,V)$ denote the unique irreducible highest weight module for $\mathfrak{g}$ corresponding to a choice of some dominant integral weight, and
let  $v\in V$ be a nonzero highest weight vector.   We set
$$V_{\mathbb{Z}}\  =\ \mathcal{U}_{\mathbb{Z}}\cdot v\,.$$
Then $V_{\mathbb{Z}}$  is a ${\mathcal U}_\mathbb{Z}$--module contained in $V_{F}=F\otimes_{\mathbb{Z}}V_{\mathbb{Z}}$.  Since $V$ is integrable, $e_{i}$ and $f_{i}$ are locally nilpotent on $V$ and $V_\Z$. It follows that they are locally nilpotent on $V_{F}$ and hence elements of $\textrm{End}(V_{F})$.
Thus for  $s,t\in F$
 and $i\in I$,  their exponentials
$$ \ \ \ \ \ \ \ u_{\alpha_i}(s)\ =\ \exp(\pi(se_i))$$
$$\text{and}~~u_{-\alpha_i}(t)\ =\ \exp(\pi(tf_i))$$
  are actually locally finite sums (meaning their action on any fixed vector is given by a finite sum), and thus define elements of  $\textrm{Aut}(V_{F})$.

For $t \in F^{\times}$ and $i\in I$  we set \[ w_{\alpha_i}(t) \  = \ u_{\alpha_i}(t)  u_{-\alpha_i}(-t^{-1}) u_{\alpha_i}(t)\] and define \[ h_{\alpha_i}(t)\ = \  w_{\alpha_i}(t) w_{\alpha_i}(1)^{-1} .\]
Each simple root $\alpha_j$ defines a character on $\{ h_{\alpha_i}(t)|t\in  F^{\times}\}$ by
\begin{equation}\label{halphaialphajchar}
  h_{\alpha_i}(t)^{\alpha_j} \ = \ t^{\langle \alpha_j,\alpha_i^\vee\rangle}.
\end{equation}
The subgroup $\langle w_{\alpha_i}(1):i\in I\rangle$ of $\textrm{Aut}(V_{F})$ g   contains a full set of Weyl group representatives.
For a real root $\alpha$
we let $u_\alpha(s)$, $s\in F$, denote a choice of corresponding one-parameter subgroup, chosen so that
\begin{equation}\label{ualphareal}
  u_\alpha(s) \ \ = \ \ w u_{\alpha_i}(\pm s)w^{-1} \in \textrm{Aut}(V_{F}) \qquad (s \in F)
\end{equation}
for   $w=w_{\beta_1}(1)\cdots w_{\beta_\ell}(1)$ and
 $\alpha=w_{\beta_1}\cdots w_{\beta_\ell}\alpha_i$, for some $i\in I$, where $\beta_1,\ldots,\beta_\ell\in\Delta$.

  We let
$$G^0_{F}\ = \ \langle u_{\alpha_i}(s), u_{-\alpha_i}(t) : s,t\in F, \ i \in I \rangle\ \subset \ \textrm{Aut}(V_F).$$
Choose a  coherently ordered basis (see \cite[\S5]{CG})   $\Psi=\{v_1,v_2,\dots\}$ of $V_\Z$, and denote by $B^0_F$ the subgroup of $G^0_F$ consisting of  elements which act upper-triangularly   with respect to $\Psi$.
For $t \in \mathbb Z_{>0}$, we let $U_t$ be the span of the $v_s \in \Psi$ for $s \le t$. Then $B^0_F U_t \subseteq U_t$ for each $t$. Let $B_t$ be the image of $B^0_F$ in $\mathrm{Aut}(U_t)$. We then have surjective homomorphisms
\[ \pi_{tt'}: B_{t'} \longrightarrow B_t, \quad t' \ge t, \] which we use to define  $B_F$ as the projective limit of the projective family $\{ B_t, \pi_{tt'} \}$.  When the field $F$ is unspecified, $B$ is to be interpreted as $B_\R$.

We define a topology on $G^0_F$ by decreeing that a base of open neighborhoods of the identity is given by the sets  \[ V_t\ = \ \{ g \in G^0_F : g v_i =v_i, \ i=1, 2, \dots , t \}. \]
Let $G_F$ be the completion of $G^0_F$ with respect to this topology.  (See \cite{CLL} for more details.)  If $R\supset \Z$ is a subring of $F$, the subgroup $G_R\subset G_F$ is defined as the stabilizer of $V_\Z \otimes_\Z R$ in $G_F$.  For future reference, we define the following subgroups of $G_F$ which play an important role in the rest of the paper:
 \begin{itemize}
 \item  $A_F=\langle h_{\alpha_i}(s) : s\in F^\times, i \in I\rangle$; and
 \item  $U_F\subset B_F$ is defined exactly as $B_F$, but with the additional stipulation that elements act unipotently upper triangularly with respect to $\Psi$.  It contains all subgroups parameterized by the $u_\alpha(\cdot)$, where $\alpha\in \Phi_+$ is a real root.  Then $B_F=U_FA_F=A_FU_F$.  When no subscript is given, $U$ is to be interpreted as $U_\R$.
 \end{itemize}
 Additionally, the following  subgroups are specific to the situation $F=\R$:
 \begin{itemize}
\item $K$ is the subgroup of $G_\R$ generated by all  $\exp(t(e_i-f_i))$, $t\in \R$ and $i\in I$ \cite{KP}; and
    \item   $A^+=\langle h_{\alpha_i}(s) : s\in\mathbb{R}_{>0}, i \in I\rangle$.  In fact, $(\mathbb{R}_{>0})^r$ can be identified with $A^+$ via the isomorphism  $(x_1,\dots, x_r)\mapsto h_{\alpha_1}(x_1)\cdots h_{\alpha_r}(x_r)$, under which $A^+$ has the Haar measure $da$ corresponding to $\prod_{i=1}^r \frac{dx_i}{x_i}$.
\end{itemize}

\begin{theorem}[\cite{DGH}]
We have the Iwasawa decomposition
\begin{equation}\label{iwasawa}
G_\R \ = \  UA^+K,
\end{equation} with uniqueness of expression.
\end{theorem}

We let $u(g)$, $a(g)$, and $k(g)$ denote the projections   from $G_{\R}$ onto  each of the respective factors in (\ref{iwasawa}).
We define the discrete group
$\Gamma = G_\Z$ as $G_\R\cap \mathrm{Aut}(V_{\mathbb{Z}})=\{\gamma\in G_\R : \gamma\cdot V_{\mathbb{Z}}= V_{\mathbb{Z}}\}.$
As in \cite{G04}, it can be shown that $(\Gamma\cap U) \backslash U$ is the projective limit of a projective family of finite-dimensional compact nil-manifolds  and thus admits
a projective limit measure $du$,  a right $U$-invariant probability measure.

\section{Convergence of the constant term}\label{constant-term}

Using the identification of $A^+$ with $(\R_{>0})^r$, each element of $\frak h^*$ gives rise to a quasicharacter of $A^+$ by (\ref{halphaialphajchar}) and linearity. Let $\lambda\in \frak h^*$ and let $\rho\in\mathfrak{h}^*$ be  the Weyl vector, which is characterized by $\langle\rho,\alpha_i^\vee\rangle =1$, $i\in I$.  We set
  \begin{equation}\label{Philambdadef}
    \aligned
    \Phi_{\lambda}:& \ G_\R\to {\mathbb C}^{\times} \\
    \Phi_{\lambda}:& \ g  \mapsto a(g)^{\lambda+\rho},
    \endaligned
  \end{equation}
   which is well-defined by the uniqueness of the Iwasawa decomposition. Clearly, $\Phi_{\lambda}$ is left $U$-invariant and right $K$-invariant.

 Let $B$ and $\Gamma$ be   as defined in Section \ref{Zform}.
 Define the Eisenstein series on $G_\R$ to be the infinite formal sum
\[
E_{\lambda}(g)\quad = \sum_{\gamma\in (\Gamma\cap{B})\backslash \Gamma}
\Phi_{\lambda}(\gamma g).
\]
Assume first that  $\lambda$ is a real linear combination of the $\alpha_i$, $i\in I$, so that $\Phi_\lambda>0$. Then we may interpret the infinite sum $E_\lambda(g)$ as a function taking values in $(0,\infty]$. Moreover, the function $E_\lambda$ may be regarded as a function on
\[
(\Gamma\cap U)\backslash G_\R/ K  \ \ \cong \ \ (\Gamma\cap U)\backslash U \times A^+
\]
by the Iwasawa decomposition (\ref{iwasawa}).
We define for all $g\in G_\R$ the so-called ``upper triangular'' constant term\footnote{This is to be distinguished from the proposal in \cite{BK} to consider constant terms in uncompleted, or so-called ``lower triangular'', parabolics.}
\[
E^\sharp_\lambda(g)\ = \ \int_{(\Gamma\cap U)\backslash U}E_\lambda(ug)du,
\]
which is left $U$-invariant and right $K$-invariant. In particular $E^\sharp_\lambda(g)$ is determined by the $A^+$-component $a(g)$ of $g$ in the Iwasawa decomposition.
Applying the Gindikin--Karpelevich formula,  a formal calculation as in \cite{G04} yields that
\begin{equation} \label{eee}
E^\sharp_\lambda(g)  \ = \ \sum_{w\in W}a(g)^{w\lambda+\rho}c(\lambda,w)\, , \ \ \text{with} \ \  c(\lambda,w) \ = \ \prod_{ \alpha \in \Phi_w}\frac{\xi(\langle\lambda, \alpha^\vee\rangle)}{\xi(1+\langle\lambda, \alpha^\vee\rangle)}\,,
\end{equation}
where
\begin{equation}\label{Phiwdef}
  \Phi_w \ = \ \Phi_+\cap w^{-1}\Phi_{-}, \ w\in W,
\end{equation}
  and $\xi(s)$ is the completed Riemann $\zeta$-function
\[
\xi(s)\ = \ \Gamma_\mathbb{R}(s)\zeta(s),
\]
where $\Gamma_\mathbb{R}(s)=\pi^{-s/2}\Gamma(s/2)$.  We shall often make use of the following explicit parametrization of $\Phi_w$, where $w$ is written as a reduced word (i.e., minimal length) $w=w_{i_1}w_{i_2}\cdots w_{i_\ell}$ in the generators $\{w_i : i\in I\}$ of $W$:
\begin{equation}\label{Phiwparam}
  \Phi_w \ \ = \ \ \{\alpha_{i_\ell},\, w_{i_\ell}\alpha_{i_{\ell-1}},\,
  w_{i_\ell}w_{i_{\ell-1}}\alpha_{i_{\ell-2}},\ldots, w_{i_\ell}\cdots w_{i_2}\alpha_{i_1}\}
\end{equation}
(see \cite[Lemma~1.3.14]{Kumar}).

Returning to (\ref{eee}), we now state an elementary estimate on the Riemann $\zeta$-function from analytic number theory.
\begin{lemma} \label{zeta} One has
\[
\lim_{\sigma\to\infty}\max_{t\in\R}\left|\frac{\xi(\sigma+it)}{\xi(\sigma+1+it)}\right|=0.
\]
\end{lemma}

\begin{proof}
It follows from the Euler product $\zeta(s)=\prod_p\frac{1}{1-p^{-s}}$ that
\[
\log\zeta(s) \ = \ - \sum_p \log(1-p^{-s}) \ = \ \sum_{n\geq 2} c_n n^{-s},
\]
where
\[
c_n \ = \ \left\{\begin{array}{ll} 1/k & \text{if }n=p^k \text{ for prime } p, \\ 0 & \text{otherwise}.\end{array}\right.
\]
In particular $|c_n|\leq 1$, hence
\[
|\log\zeta(\sigma+it)| \ \leq \
\sum_{n\geq 2}|c_n| n^{-\sigma} \ \leq  \ \sum_{n\geq 2} n^{-\sigma} \ \longrightarrow \  0,\quad \text{as}\quad \sigma\to\infty,
\]
e.g., by dominated convergence.
Therefore $\lim\limits_{\sigma\to\infty}\zeta(\sigma+it)=1$ uniformly in $t\in \R$. On the other hand,  it is easy to see  from standard properties of $\Gamma$-functions (or  applying dominated convergence to the integral formula \eqref{gk}) that
\[
\lim_{\sigma\to\infty}\frac{\Gamma_\mathbb{R}(\sigma+it)}{\Gamma_\mathbb{R}(\sigma+it+1)} \ = \ 0,
\]
again uniformly in $t\in \R$.  Thus the Lemma follows by multiplying these two estimates.
 \end{proof}

Let $\mathcal{C} \subset\mathfrak{h}$ be the open fundamental chamber
\[
\mathcal{C} \ = \ \{x\in \mathfrak{h}: \langle\alpha_i, x\rangle>0,~  i\in I\}.
\]
   Let $\mathfrak{C}$ denote the interior of the Tits cone $\mathop{\cup}_{w \in W}  w\overline{  \mathcal C}$   corresponding to $\mathcal{C}$ \cite[\S3.12]{K}.  Using the exponential map $\exp: \mathfrak{h}\to A^+$, set $A_\mathcal{C}=\exp \mathcal{C}$ and $A_\mathfrak{C}=\exp \mathfrak{C}$.
Let
\begin{equation} \label{C-star}
\mathcal{C}^* \ = \ \{\lambda\in\mathfrak{h}^*: \langle\lambda,\alpha_i^\vee\rangle>0, ~ i\in I\}.
\end{equation}
Let $\mathcal{K}$ be a compact subset of $\mathfrak{C}$ and $\mu \in P \cap \mathcal{C}^*$. We define  $A_{\mathcal{K},\mu}(N)$ to be the number of $\mu'$ in the Weyl orbit $W\cdot \{\mu\}$ whose maximum on $\mathcal{K}$ is $\geq -N$.  We now recall a lemma due to Looijenga \cite{Lo} (which is contained in   Lemma 3.2 and the beginning of the proof of Proposition 3.4 there -- see also \cite[\S10.6]{K}).

\begin{lemma}\label{lo} We have  $A_{\mathcal{K},\mu}(N)=O(N^r)$ as $N\to\infty$.
Furthermore, for $\lambda\in\mathfrak{h}^*_\mathbb{C}$ with $\operatorname{Re}(\lambda)\in \mathcal{C}^*$,  $\sum_{w\in W} a^{w\lambda}$ converges absolutely and uniformly for   $a$ in any fixed compact subset of $A_\mathfrak{C}$.
\end{lemma}

We apply Lemma~\ref{zeta} and Lemma~\ref{lo} to show the following convergence result:

\begin{theorem} \label{3.3}
If   $\lambda\in\mathfrak{h}^*_\mathbb{C}$ satisfies $\operatorname{Re}(\lambda  - \rho) \in \mathcal{C}^*$, then
$E^\sharp_\lambda(g)$ given in \eqref{eee} converges absolutely for $g\in UA_\mathfrak{C}K$, and in fact uniformly for $a(g)$ lying in any fixed compact subset of $A_\mathfrak{C}$.
\end{theorem}

\begin{proof}
By applying absolute values, we may assume without loss of generality that $\lambda$ is real.
Let $S \ge 1$ be a constant such that $\left|\frac{\xi(s)}{\xi(s+1)}\right|\leq 1$ when $\operatorname{Re}(s)>S$, which exists by Lemma \ref{zeta}. Since $\langle\lambda, \alpha_i^\vee\rangle>1$ for all $i\in I$, we see that
$\langle\lambda,\alpha^\vee\rangle>S$ for all but finitely many positive roots $\alpha$.
 Therefore $c(\lambda,w)$ is bounded in $w\in W$, and  the convergence follows from that of $\sum_{w\in W}a^{w\lambda}$ in Lemma~\ref{lo}.
\end{proof}

By applying Tonelli's theorem as in \cite[\S9]{G04}, we obtain the following result for arbitrary Kac--Moody groups.

\begin{corollary}\label{cor}
For $\lambda\in\mathfrak{h}^*_\mathbb{C}$ with $\operatorname{Re}(\lambda -\rho) \in \mathcal{C}^*$ and any  compact subset $\mathfrak{S}$ of $A_\mathfrak{C}$, there exists a measure zero subset $S_0$ of $(\Gamma\cap U)\backslash U\mathfrak{S}$ such that the series $E_\lambda(g)$ converges absolutely for $g\in U\mathfrak{S}K$ off the set $S_0K$.
\end{corollary}

For later use, we further strengthen Lemma~\ref{lo} and prove that

\begin{theorem} \label{mainthm}
Assume that $\lambda\in\mathfrak{h}^*_\mathbb{C}$ with $\operatorname{Re}(\lambda)\in \mathcal{C}^*$. Then for any $M>0$,
\[
\sum_{w\in W} M^{\ell(w)} a^{w\lambda}
\]
converges absolutely and uniformly for $a$ in any fixed compact subset of $A_\mathfrak{C}$.
\end{theorem}

\begin{proof}
We take advantage of the fact that the constraints on both $\lambda$ and $a$ are preserved under small perturbations.  Let $s>0$ be sufficiently small so that $\lambda':=\lambda-s\rho\in \mathcal{C}^*$. Then
\[
M^{\ell(w)}a^{w\lambda} \ = \ M^{\ell(w)}a^{ws\rho+w\lambda'} \ \leq \  \frac{1}{2} \left(M^{2\ell(w)}a^{2ws\rho}+a^{2 w\lambda'}\right).
\]
Since the convergence of $\sum_{w\in W}a^{2 w\lambda'}$  is handled  by Lemma~\ref{lo}, it suffices to prove the absolute and uniform convergence of
\[
\sum_{w\in W}M^{2\ell(w)}a^{2ws\rho} \ =  \sum_{w\in W}T^{\ell(w)} a^{2ws\rho}\,, \ \ \ T=M^2\,,
\]
for $a$ in any fixed compact subset of $A_\mathfrak{C}$, for any $T, s>0$.

Using the fact that
\begin{equation}\label{rhominuswrho}
 \rho-w\rho  \ = \ \sum_{\alpha\in \Phi_{w^{-1}}} \alpha
\end{equation}
is a sum of $\ell(w)$ positive roots (see (\ref{Phiwparam}) and \cite[1.3.22(3)]{Kumar}), we may write
  \[
 T^{\ell(w)} a^{2ws\rho}\  = \ a^{2s\rho} \prod_{\alpha\in \Phi_{w^{-1}}} Ta^{-2s\alpha}
 \  = \ a^{2s\rho} \prod_{\alpha\in \Phi_{w^{-1}}} (Ta^{-s\alpha})a^{-s\alpha}\,.
 \]
 We shall prove that $a^{s\alpha}>T$ for all but finitely many positive  roots $\alpha$ (in particular, those in $\Phi_{w^{-1}}$).   Assuming this momentarily, the parenthetical term $Ta^{-s\alpha}$ is bounded, and can only be greater than 1 for finitely many  positive roots $\alpha$.   Hence there exists a constant $D>0$ depending continuously on $a, s$ and $T$ such that
 \[
  T^{\ell(w)} a^{2ws\rho} \ \leq  \ D a^{2s\rho}\prod_{\alpha\in \Phi_{w^{-1}}} a^{-s\alpha} \ = \  Da^{2s\rho} a^{s(w\rho-\rho)} \ = \  Da^{s(w\rho+\rho)}\,.
 \]
 Hence
 \[
 \sum_{w\in W}T^{\ell(w)} a^{2ws\rho}\ \leq \  Da^{s\rho}\sum_{w\in W} a^{w(s\rho)}\,,
 \]
 and the Theorem follows from a second application of Lemma~\ref{lo}.

Finally, we return to the claim that $a^{s\alpha}>T$ for all but finitely many positive  roots $\alpha$.  Let $H_\rho$ denote the unique element of $\frak h$ such that $\langle \alpha_i,H_\rho\rangle = \langle \rho,\alpha_i^\vee \rangle = 1$
for all $i\in I$ (it exists by the assumed nondegeneracy of the Cartan matrix).
Since $a$ is assumed to lie in the open set $A_{\frak C}$, there exists some $\varepsilon>0$ such that $a=a_1a_2$, with $a_1$ also an element of  $A_{\frak C}$ and $a_2=e^{\varepsilon H_\rho}$.  Then
\begin{equation}\label{aa1a2}
  a^{s\alpha} \  = \ a_1^{s\alpha}\,a_2^{s\alpha} \ = \ a_1^{s\alpha}\, e^{\varepsilon s\langle \alpha,H_\rho\rangle}\,.
\end{equation}
  Since $a_1\in A_{\frak C}$, the values of $a_1^{\alpha}$ are at least 1 for all but finitely many $\alpha$ \cite[Prop.~3.12(c)]{K}.  At the same time, $\langle \alpha,H_\rho\rangle$ is the sum of the (nonnegative) coefficients of $\alpha$ when expanded as a linear combination of simple roots, and thus tends to infinity as $\alpha$ varies.  The claim now follows from (\ref{aa1a2}).
\end{proof}

 \section{Everywhere convergence of Eisenstein series} \label{Conv}

In the previous section we demonstrated the almost-everywhere absolute convergence of Eisenstein series.
In this section we will prove the absolute convergence of the Eisenstein series on $\Gamma UA_{\mathfrak C}K$ under Godement's criterion on the spectral parameter $\lambda$, subject to a condition (Property~\ref{conj}) on the root system of the Kac--Moody group. The key idea is that the Eisenstein series can be nearly bounded  by its constant term. 

We start with some    calculations for the group $SL(2,\mathbb{R})$.  It follows from direct computation that
\begin{equation}\label{asl2u}
a \left ( \begin{pmatrix} 0 & -1 \\ 1 & u \end{pmatrix}\right )^{\alpha}
 \ = \ \begin{pmatrix} \frac{1}{\sqrt{1+u^2}} & 0 \\ 0 & \sqrt{1+u^2}\end{pmatrix}^\alpha
\ = \ \frac{1}{1+u^2}\,,
\end{equation}
where $\alpha$ is the positive simple root of the diagonal Cartan.
Define
\begin{equation}\label{gk}
\aligned
c_\infty(s) \ & := \ \int_\mathbb{R}a \left ( \begin{pmatrix} 0 & -1 \\ 1 & u\end{pmatrix}\right )^{(s+1)\alpha/2}du  \\ & = \ \int_\mathbb{R}(1+u^2)^{-\frac{s+1}{2}}du \ = \ \frac{\Gamma_\mathbb{R}(s)}{\Gamma_\mathbb{R}(s+1)}\,, \  \  \operatorname{Re}(s)>0.
\endaligned
\end{equation}
Returning to the setting of a general Kac--Moody group $G_\R$, we  shall now assume that $\lambda\in \frak h^*_\C$ is actually real, i.e., $\lambda\in \frak h^*$, since this entails no loss of generality in considering absolute convergence.
Recall the notation   $u_\alpha(x)$ from  (\ref{ualphareal}).

\begin{lemma} \label{4.1}
Assume that  $\alpha$ is a positive simple root such that $\langle\lambda, \alpha^\vee\rangle>0$ (equivalently, $\langle \lambda+\rho,\alpha^\vee\rangle > 1)$. For any $x\in\mathbb{R}$ and $g=uak \in G_\mathbb{R}$, with $u\in U$, $a\in A^+$, and $k\in K$, we have
\[
\sum_{m\in\mathbb{Z}}a(w_\alpha u_\alpha(x+m)g)^{\lambda+\rho} \ \leq \  a^{w_\alpha(\lambda+\rho)}\left(2+a^\alpha c_\infty(\langle\lambda,\alpha^\vee \rangle)\right).
\]
In particular, for any $\epsilon>0$ there exists a constant $M=M_\epsilon>2$ such that
\begin{equation}\label{4.2new}
\sum_{m\in\mathbb{Z}}a(w_\alpha u_\alpha(x+m)g)^{\lambda+\rho} \ \leq \  M a^{w_\alpha(\lambda+\rho)}(1+a^\alpha)
\end{equation}
whenever $\langle \lambda,\alpha^\vee \rangle \ge \epsilon$, uniformly over $x\in\R$.
\end{lemma}

\begin{proof}
Factor $u=u_\alpha(y)u^{(\alpha)}$, where $y\in\R$ and $u^{(\alpha)}$ lies in the unipotent radical $U^{(\alpha)}$ of the parabolic subgroup whose Levi component is generated by $A$ and
the one parameter subgroups $u_{\pm\alpha}(\cdot)$. Then since $w_\alpha u_\alpha(x+m+y)$ lies inside this Levi, it normalizes $U^{(\alpha)}$. Hence we may write
\begin{align*}
a(w_\alpha u_\alpha(x+m)g) \ = \ &~a(w_\alpha u_\alpha(x+m+y)u^{(\alpha)}a k)\\
\ = \ &~ a(w_\alpha u_\alpha(x+m+y)a)\,,
\end{align*}
since $U^{(\alpha)}\subset U$.
As $x$ is arbitrary and plays no role in the asserted upper bounds, we may replace $x$ by $x-y$ so that we are instead considering the sum
\begin{equation} \label{sum}
\sum_{m\in\mathbb{Z}} a(w_\alpha u_\alpha(x+m)a)^{\lambda+\rho} \ = \ a^{w_\alpha(\lambda+\rho)}\sum_{m\in\mathbb{Z}}a(w_\alpha u_\alpha(a^{-\alpha}(x+m)))^{\lambda+\rho}.
\end{equation}
Recalling (\ref{asl2u}),
the sum (\ref{sum}) becomes
\[
a^{w_\alpha(\lambda+\rho)}\sum_{m\in\mathbb{Z}}(1+a^{-2\alpha}(x+m)^2)^{-\frac{1}{2}\langle\lambda+\rho, \alpha^\vee\rangle}.
\]
Since $\langle \lambda+\rho,\alpha^\vee\rangle =\langle\lambda,\alpha^\vee\rangle+1$, comparing the sum to its corresponding integral gives  the estimate
\[
\aligned
\sum_{m\in\mathbb{Z}}(1+a^{-2\alpha}(x+m)^2)^{-\frac{s+1}{2}} \ & \leq \  2+\int_\mathbb{R}(1+a^{-2\alpha}x^2)^{-\frac{s+1}{2}}dx \\ & = \ 2+a^\alpha\, c_\infty(s)\,,
\endaligned
\]
for $s=\langle \lambda,\alpha^\vee\rangle>0$.  The second statement (\ref{4.2new}) for $M=  \max(2,c_\infty(\epsilon))$ now follows  from the bound  of $c_\infty(\langle\lambda,\alpha^\vee\rangle)\le c_\infty(\epsilon)$.
\end{proof}
%
%

We wish to generalize (\ref{4.2new}) to arbitrary $w\in W$ by induction on $\ell(w)$, which is the most complicated part of the argument.
For technical reasons we need to introduce the following property.   Recall  the notation $\Phi_w=\Phi_+\cap w^{-1}\Phi_{-}$ from (\ref{Phiwdef})-(\ref{Phiwparam}).

\begin{property}\label{conj}
Every nontrivial $w \in W$ can be written as
$w=vw_{\beta}$, where $\beta$ is a positive simple root, $\ell(v)<\ell(w)$, and $\alpha -\beta$ is never a real root for any $\alpha \in \Phi_v$.
\end{property}

It is not hard to see (e.g., see \cite[\S 3.2]{T}) that   Property~\ref{conj} is unchanged if the word ``real'' is omitted.  This is because if $\alpha-\beta=\alpha+w_{\beta}\beta$ is a (necessarily, positive) root, then $v(\alpha-\beta)=v\alpha+w\beta$ is a negative root (cf.~(\ref{Phiwparam})), and all elements of $\Phi_v$ are real roots (again appealing to (\ref{Phiwparam})).

Although Property \ref{conj} does not hold for  arbitrary Kac--Moody root systems (see Example \ref{counterex} for counterexamples), we shall nevertheless demonstrate that it holds in infinitely many examples.

\begin{proposition} \label{prop-generic}
Property \ref{conj} holds if  the Cartan matrix  ${\sf A}=(a_{ij})$ of $\mathfrak g$ is symmetric and  $|a_{ij}|\ge 2$ for all $i,j \in I$.
\end{proposition}

\begin{proof}
First we claim that if
 $w_{i_1}\cdots w_{i_k}$ is a reduced word and the  (positive, by (\ref{Phiwparam}))  root $w_{i_k} \cdots w_{i_2} \alpha_{i_1}$ is expanded as $\sum_{j \in I} m_j \alpha_j$, then  $m_{i_k} > m_j$ for $j \neq i_k$.  Indeed, this can be seen as follows by induction, the case $k=2$ being a consequence of our assumption.
  Assume that the claim is true for $k$. Then we write $w_{i_{k+1}}w_{i_k} \cdots w_{i_2} \alpha_{i_1} = \sum_{j \in I} m_j' \alpha_j$ and obtain
\begin{align*}  \sum_{j \in I} m_j' \alpha_j &= w_{i_{k+1}}\sum_{j \in I} m_j \alpha_j \\ &= m_{i_k}(\alpha_{i_k}-a_{i_{k+1}i_k} \alpha_{i_{k+1}}) \,-\, m_{i_{k+1}} \alpha_{i_{k+1}} \\ &\qquad + \sum_{j \neq i_k, i_{k+1}} m_j (\alpha_j - a_{i_{k+1}j} \alpha_{i_{k+1}})\,.
\end{align*}
We have  $m'_j=m_j <m_{i_k}=m'_{i_k}$ for $j \neq i_k, i_{k+1}$ and
\[ m'_{i_{k+1}} \ = \ -m_{i_k}a_{i_{k+1}i_k}\,-\,m_{i_{k+1}} \,-\, \sum_{j \neq i_k, i_{k+1}} m_j a_{i_{k+1}j} \ > \ m_{i_k} \ = \ m'_{i_k}\,,\]
since $-(1+a_{i_{k+1}i_k})m_{i_k}-m_{i_{k+1}}\ge m_{i_k}-m_{i_{k+1}}>0$. This proves the claim.

We now claim that for any reduced word $w=w_{i_1}\cdots w_{i_\ell}$, one has $\langle \alpha, \alpha_{i_\ell}^\vee \rangle <0$ for each $\alpha \in \Phi_v$, where $v=w_{i_1}\cdots w_{i_{\ell-1}}$.
By (\ref{Phiwparam}) the elements of $\Phi_v$  can be written as $w_{i_{\ell-1}} \cdots w_{i_{k+1}}\alpha_{i_k}$, $k=1,\ldots, \ell-1$. By the above claim, if  $w_{i_{\ell-1}} \cdots w_{i_{k+1}}\alpha_{i_k}= \sum_{j \in I} m_j \alpha_j$, then $m_{i_{\ell-1}}>m_j$ for $j \neq i_{\ell-1}$. In particular, $m_{i_{\ell-1}}>m_{i_\ell}$ since  $i_\ell \neq i_{\ell-1}$ for a reduced word.  Therefore
\begin{equation} \label{ge2}
\aligned
\langle w_{i_{\ell-1}} \cdots w_{i_{k+1}}\alpha_{i_k}, \alpha_{i_{\ell}}^\vee \rangle & \  = \  \langle \sum_{j \in I} m_j \alpha_j , \alpha_{i_{\ell}}^\vee \rangle  \ = \
 \sum_{j \neq i_{\ell}} m_j a_{i_\ell j} + 2 m_{i_{\ell}} \\
& \ \le \  m_{i_{\ell-1}} a_{i_{\ell}i_{\ell-1}} \,+ \,2 m_{i_{\ell}} \  <  \ 0
\endaligned \end{equation}
after dropping all but the $j=i_{\ell-1}$ term from the sum.

Let $(\cdot | \cdot )$  denote the  symmetric bilinear form on $\mathfrak h^*$ defined in \cite[\S2.3]{K}, which in general involves a choice of scaling by a positive constant; since ${\sf A}$ is symmetric it can be written as  $(\alpha_i | \alpha_j)=a_{ij}= \langle \alpha_j , \alpha_i^\vee \rangle$.  The bilinear form $(\cdot|\cdot)$ is Weyl invariant \cite[Prop.~3.9]{K}, and so one has $(\beta | \beta)=(\alpha_i|\alpha_i)$ for any root of the form $\beta=w\alpha_i$, for some $w\in W$ and $i\in I$.  On the other hand, \cite[Prop.~5.2(c)]{K} shows that $(\beta | \beta)\le 0$ for any imaginary root (that is, a root which is not a Weyl translate of a simple root).  Thus in any event one has
\begin{equation}\label{betabetale2}
  (\beta|\beta) \ \le \  \max_{i\in I}\,(\alpha_i|\alpha_i)
\end{equation}
for any root $\beta$.\footnote{For later reference we remark that even though $A$ here is assumed to be symmetric, the conclusion (\ref{betabetale2}) nevertheless holds even if $A$ is merely symmetrizable (with $(\cdot|\cdot)$ as defined  in \cite[\S2.3]{K}, or any positive scalar multiple of it).}

 Let $\alpha = w_{i_{\ell-1}} \cdots w_{i_{k+1}}\alpha_{i_k}$ be an element of $\Phi_v$, where $1 \le k  \le \ell-1$.  Since
\begin{equation}\label{phivandw}
  \Phi_w \ = \ w_{i_\ell}\Phi_v\cup\{\alpha_{i_\ell}\}
\end{equation}
by (\ref{Phiwparam}),
$\alpha$ cannot equal $\alpha_{i_\ell}$, for if it did then $-\alpha_{i_\ell}=w_{i_\ell}\alpha\in\Phi_w$ would be  a positive root.
We compute
\[
\aligned
(\alpha - \alpha_{i_{\ell}}| \alpha- \alpha_{i_{\ell}} ) & \ = \  (\alpha_{i_k}| \alpha_{i_k})  \, + \,  (\alpha_{i_\ell}| \alpha_{i_\ell}) \, -\,  2 (\alpha| \alpha_{i_{\ell}}) \\ & \ = \ 4 \, -\, 2 \langle \alpha, \alpha_{i_{\ell}}^\vee \rangle \  > \  4
\endaligned
\]
using the second claim above.
Thus $\alpha - \alpha_{i_\ell}$ is not a root for $\alpha \in \Phi_v$.
\end{proof}

Unfortunately Proposition~\ref{prop-generic}'s restrictive assumption on $|a_{ij}|$ rules out many cases of interest.  In fact, only one hyperbolic Kac--Moody algebra of rank $\ge 3$ satisfies its hypotheses, namely the one having ${\sf A} =\left(\begin{smallmatrix} \,2 & -2 & -2 \\
-2 & \,2 & -2 \\
-2 & -2 & \,2 \\
\end{smallmatrix}\right)$.  In contrast, however, the assumptions of Proposition~\ref{prop-generic} hold for  many  rank $2$ hyperbolic root systems, even ones with non-symmetric Cartan matrices:

\begin{proposition} \label{prop-rank2}
Let ${\sf A}=\begin{pmatrix} 2 & -b \\ -a & 2\end{pmatrix}$, $a,b \ge 2$ and $ab \ge 5$. Then Property \ref{conj} holds for the Kac--Moody root system associated to ${\sf A}$.
\end{proposition}

\begin{proof}
The Weyl group of the root system is the infinite dihedral group generated by the simple reflections $w_1$ and $w_2$.
In terms of the basis $\{\alpha_1,\alpha_2\}$ of $\mathfrak{h}^*$, $w_1$ acts by the matrix $\left(\begin{smallmatrix} -1 & b \\ 0 & 1\end{smallmatrix}\right)$ and $w_2$ acts by the matrix $\left(\begin{smallmatrix} 1& 0 \\ a & -1\end{smallmatrix}\right)$.
Since $w_1$ and $w_2$ play symmetric roles, it suffices to establish Property \ref{conj} in the case that $\beta=\alpha_2$.  Thus $w=vw_\beta$ is represented by a reduced word ending in $w_2$, and hence $v$ is represented by a reduced word ending in $w_1$.  In particular, $v$ must have the form $(w_2 w_1)^m$ or $w_1(w_2 w_1)^m$ for some $m\ge 0$.  The   roots $\alpha\in \Phi_v$   either have the form
$$
(w_1w_2)^n \alpha_1 \ \ \text{or} \ \ (w_1w_2)^n w_1 \alpha_2\,,
$$
for some $n\ge 0$ by (\ref{Phiwparam}).  To establish Property \ref{conj}, we show that none of
\begin{equation}\label{w1w2nsroot}
 (w_1w_2)^n \alpha_1 -\alpha_2 \ \ \text{ or } \ \  (w_1w_2)^n w_1 \alpha_2 -\alpha_2\,,  \ \  \ n\ge 0,
\end{equation}
 are roots.  Following the strategy in the proof of Proposition~\ref{prop-generic}, we will use the symmetric bilinear form $(\cdot|\cdot)$ on $\mathfrak h^*$ associated to the symmetrized matrix  $\begin{pmatrix} a & 0 \\ 0 & b
\end{pmatrix} {\sf A} = \begin{pmatrix} 2a & -ab \\ -ab & 2b\end{pmatrix}$, and appeal to condition (\ref{betabetale2}), which on our context states that
 \begin{equation}\label{betabetaleab}
   (\beta|\beta) \le \max\{2a,2b\}
 \end{equation}
for any root $\beta$.

Let $\mu= \frac{\sqrt{ab} +\sqrt{ab-4}}2>1$ and
$\displaystyle{h_n= \frac 1{\mu-\mu^{-1}}(\mu^n-\mu^{-n})}$, which is a monotonically increasing function of $n$ satisfying
$$h_{n+2}=h_n+\mu^{n+1}+\mu^{-n-1}.$$
Since $\sqrt{ab}>2$ we have for any $n\ge 1$ that
\begin{align*} \sqrt{ab}  \, h_{n+2} - 2 h_{n+1} & = \sqrt{ab} \, (\mu^{n+1} + \mu^{-n-1}) -2 (\mu^{n} + \mu^{-n}) + \sqrt{ab}  \, h_{n} - 2 h_{n-1} \\ & > \sqrt{ab}  \, h_{n} - 2 h_{n-1}.
 \end{align*}
Repeating,  we obtain
\begin{equation}\label{h2nlowerbounds}
\aligned
  \sqrt{ab}  \, h_{2n+2} - 2 h_{2n+1} & \  \ge \  \sqrt{ab}  \, h_{2} - 2 h_{1}   \  = \    a b -2  \ \ \  \text{and}\\
   \sqrt{ab}  \, h_{2n+1} - 2 h_{2n} &  \ \ge \  \sqrt{ab}  \, h_{1} - 2 h_{0}  \ = \  \sqrt{a b} > 2  \\
\endaligned
\end{equation}
for any $n\ge 0$.

It follows from the formulas for the Weyl action that
\begin{align*}
(w_1w_2)^n\alpha_1& \ = \ h_{2n+1}\alpha_1 \,+\,\tfrac{\sqrt a}{\sqrt b} \ h_{2n} \alpha_2,\\
(w_1w_2)^n w_1\alpha_2& \ = \  \tfrac{\sqrt b}{\sqrt a} \,  h_{2n+2}\alpha_1\,+\,h_{2n+1} \alpha_2
\end{align*}
(the case $n=0$ is obvious, and both sides satisfy the same recurrence relations in $n$).
Since $( \cdot | \cdot)$ is $W$-invariant and $\sqrt{ab}>2$, we compute
\[
\begin{split}
 \left((w_1w_2)^n \alpha_1   -   \alpha_2 \right. & \left. |\, (w_1w_2)^n \alpha_1 -\alpha_2 \right)  \\
&  = \  (\alpha_1 | \alpha_1) + (\alpha_2 | \alpha_2)  -2 ((w_1w_2)^n \alpha_1 |\alpha_2) \\
& = \ 2a+2b-2 \left(h_{2n+1}\alpha_1 +\tfrac{\sqrt a}{\sqrt b} \ h_{2n} \alpha_2 \Big | \alpha_2 \right) \\
& = \  2a+2b +2ab \, h_{2n+1}-4 \sqrt{ab} \, h_{2n} \\
 & = \  2a+2b +2 \sqrt{ab} ( \sqrt{ab} \, h_{2n+1}-2 \, h_{2n}) \\
 &  > \  2a+2b \  > \  \max \{ 2a, 2b \}\,,\\
\end{split}
\]
so that  $(w_1w_2)^n \alpha_1 -\alpha_2$ cannot be a root according to (\ref{betabetaleab}).

Similarly,
\[
\begin{split}  \left((w_1w_2)^n w_1 \alpha_2 -\alpha_2\right. & \left. |\, (w_1w_2)^n w_1 \alpha_2 -\alpha_2 \right ) \\
&  = \  2(\alpha_2 | \alpha_2)   -2 ((w_1w_2)^n w_1 \alpha_2 |\alpha_2) \\
&= \ 4b-2 \left(\tfrac{\sqrt b}{\sqrt a} \  h_{2n+2}\alpha_1+h_{2n+1} \alpha_2 \Big | \alpha_2 \right) \\
&= \  4b +2 \tfrac{\sqrt b}{\sqrt a} ab\  h_{2n+2} -4b\, h_{2n+1}
 \\ &= \  4b + 2b \left ( \sqrt{ab}\, h_{2n+2} -2 h_{2n+1} \right ) \\
& \ge \  4b+2b(ab-2)  \ = \  2ab^2 \  >  \ \max \{ 2a, 2b \} \,,
 \end{split}
 \]
where we have used  the assumption $a,b \ge 2$.
Thus $(w_1w_2)^n w_1 \alpha_2 -\alpha_2$ is not a root either.
\end{proof}



\begin{ex} \label{counterex}
(1) Let $\Phi$ be the root system for the (asymmetric) Cartan matrix  ${\sf A} =\begin{pmatrix} 2 & -1 \\ -a & 2\end{pmatrix}$, $a \ge 5$ (so that $\det {\sf A} <0$).  Consider  $w = w_{\alpha_2}w_{\alpha_1}w_{\alpha_2}$, which is the unique reduced for $w$ of minimal length.  Letting $v= w_{\alpha_2}w_{\alpha_1}$ and $\beta=\alpha_2$, we see   $w_{\alpha_1} \alpha_2 \in \Phi_v$ and    $w_{\alpha_1} \alpha_2 -\alpha_2 = \alpha_1$, hence Property \ref{conj} does not hold for  this Kac-Moody root system.

(2) Let $\Phi$ be the  $A_2$ root system and consider  $w = w_{\alpha_2}w_{\alpha_1}w_{\alpha_2}=w_{\alpha_1}w_{\alpha_2}w_{\alpha_1}$.  Then Property \ref{conj} does not hold for $w$, regardless of how it is written as a reduced word.

\end{ex}

\medskip

In the rest of this section  we will establish our main result on convergence (which assumes  Property \ref{conj}).

\begin{definition} \label{def-adm}
Assume that Property \ref{conj} holds. Then by recursion any nontrivial $w\in W$ can be expressed as a reduced word $w=w_{\beta_1}\cdots w_{\beta_\ell}$  such that \[ \alpha - \beta_{i+1}  \text{ is never a root for any }
\alpha\in\Phi_{v_i},  \quad i =1,2, \dots , \ell-1 ,\]
where $v_i=w_{\beta_1}\cdots w_{\beta_i}$. Such a reduced word will be called an {\em admissible word} for $w$.
\end{definition}

\begin{lemma} \label{lem-aqa}
Assume Property \ref{conj} and that $\lambda \in \mathcal{C}^*$. Let $w=w_{\beta_1}\cdots w_{\beta_\ell}$ be an admissible word for $w \in W$ and again set $v_i=w_{\beta_1}\cdots w_{\beta_i}$. Then
\begin{equation}\label{ineq}
\langle v_i^{-1}(\lambda+\rho)+\sum_{\alpha\in S_i}\alpha, \beta_{i+1}^\vee\rangle \ > \ 1, \quad i = 1, 2, \dots , \ell-1,
\end{equation}
for any  subset $S_i$ of $\Phi_{v_i}$.
\end{lemma}

\begin{proof}
For $\alpha  \in \Phi_{v_i}$, consider the root string $\alpha + m \beta_{i+1}$, $m \in \mathbb Z$. Since $\alpha$ and $\beta_i$ are real roots and $\alpha - \beta_{i+1}$ is not a root by Property \ref{conj}, we have by \cite[Prop.~5.1(c)]{K} that
\[ \langle \alpha , \beta_{i+1}^\vee \rangle \   = \  - \max \, \{ m : \alpha + m \beta_{i+1} \text{ is a root}  \}  \ \le \  0\, . \]
Note that (\ref{ineq}) is equivalent to
\[
\langle v_i^{-1}\lambda -\sum_{\alpha\in \Phi_{v_i} \setminus S_i} \alpha, \beta_{i+1}^\vee\rangle \  > \ 0
\]
by (\ref{rhominuswrho}).  Since $\lambda \in \mathcal{C}^*$ and the fact that $v_i\beta_{i+1}^\vee$ is a positive root (cf.~(\ref{Phiwparam})), we have
\[
\langle v_i^{-1}\lambda -\sum_{\alpha\in \Phi_{v_i} \setminus S_i} \alpha, \beta_{i+1}^\vee\rangle   \ \ge \  \langle v_i^{-1}\lambda,\beta_{i+1}^\vee\rangle \ = \ \langle\lambda, v_i\beta_{i+1}^\vee\rangle \ > \ 0,
\]
and the claim follows.
\end{proof}

Let $U_w\subset U$ be the subgroup generated by the one-parameter subgroups for the roots
in $\Phi_w$.

\begin{lemma} \label{lem-aqq}
Assume Property \ref{conj}  and let $w=w_{\beta_1}\cdots w_{\beta_\ell}$ be an admissible word for an element $w \in W$ (see Definition \ref{def-adm}). Suppose that \[ u = u_{\nu_1}(x_1) \cdots u_{\nu_\ell}(x_\ell) \in U_w\ \text{ and } \ \gamma = u_{\nu_1}(m_1) \cdots u_{\nu_\ell}(m_\ell) \in \Gamma \cap U_w \] for $x_1,\ldots, x_\ell\in\mathbb{R}$ and $m_1, \ldots, m_\ell \in \mathbb Z$, where $\nu_i = w_{\beta_\ell} \cdots w_{\beta_{i+1}}\beta_i \in \Phi_w$, $i=1, \dots , \ell$ (cf.~(\ref{Phiwparam})).
Then
\[wu\gamma \  = \   w_{\beta_1} u_{\beta_1}(\epsilon_1(x_1+m_1))\cdots w_{\beta_\ell}u_{\beta_\ell}(\epsilon_{\ell}(x_\ell+m_\ell)) \]
for some  signs $\epsilon_1,\ldots,\epsilon_\ell\in\{-1,1\}$ depending only on $\beta_1,\ldots,\beta_\ell$ and the choice of Weyl group reprsentatives in $\textrm{Aut}(V_{F})$.
\end{lemma}
\begin{proof}
We will show that the sum of the positive real roots $\nu_i$ and $\nu_j$  is not a root, which implies that their root spaces and hence one-parameter subgroups   $u_{\nu_i}(\cdot)$ and $u_{\nu_j}(\cdot)$ commute.
The Lemma then follows immediately from this commutativity, since
$$u\gamma \ = \  u_{\nu_1}(x_1+m_1) \cdots u_{\nu_\ell}(x_\ell+m_\ell)$$
and one can apply (\ref{ualphareal}).

After  multiplying   on the left by $w_{\beta_j}\cdots w_{\beta_\ell}$, the claim that $\nu_i+\nu_j$ is not a root is equivalent to the statement  that
$$w_{\beta_{j-1}}\cdots w_{\beta_{i+1}}\beta_i \ - \ \beta_j\,,\qquad\qquad 1\le i < j \le \ell,$$
is never a root.  This follows from Definition~\ref{def-adm} since
$$w_{\beta_{j-1}}\cdots w_{\beta_{i+1}}\beta_i \ \in  \ \Phi_{v_{j-1}},$$
 where $v_{j-1}=w_{\beta_1}\cdots w_{\beta_{j-1}}$ satisfies
$$v_{j-1}w_{\beta_{j-1}}\cdots w_{\beta_{i+1}}\beta_i  \ = \
w_{\beta_1}\cdots w_{\beta_{i-1}}(w_{\beta_i}\beta_i) \ = \
-w_{\beta_1}\cdots w_{\beta_{i-1}}\beta_i$$
and we have used (\ref{Phiwparam}).
\end{proof}

We can now prove the convergence of Eisenstein series on Kac–Moody groups satisfying   Property \ref{conj}, after establishing one more lemma.

\begin{lemma}\label{4.5}
Assume  that $\lambda \in \mathcal{C}^*$ and   Property \ref{conj}. Then there exists a constant $M>0$ depending continuously on $\lambda$ such that, for an admissible word $w=w_{\beta_1}\ldots w_{\beta_\ell}$,
 \begin{multline*}
\sum_{m_1,\ldots, m_\ell\in\mathbb{Z}}a(w_{\beta_1} u_{\beta_1}(x_1+m_1)\cdots w_{\beta_\ell}u_{\beta_\ell}(x_\ell+m_\ell)g)^{\lambda+\rho}
\\ \leq \  M^{\ell}
a^{w^{-1}(\lambda+\rho)}\prod_{\alpha \in \Phi_w}(1+a^\alpha),
\end{multline*}
uniformly for $x_1,\ldots, x_\ell\in\mathbb{R}$, $g\in G_\mathbb{R}$, and $a=a(g)$.

\end{lemma}

\begin{proof}

Since $\lambda$ is in the open chamber $\mathcal C^*$, there exists a positive constant $\epsilon>0$ such that $\lambda- \epsilon \rho \in \mathcal C^*$.  In particular, $\langle \lambda,\alpha^\vee\rangle \ge \epsilon$ for any positive root $\alpha$.  We will  prove the lemma with the value $M=  \max(2,c_\infty(\epsilon))$ from the proof of Lemma~\ref{4.1}, using an induction on $\ell=\ell(w)$.  Indeed,
the case $\ell=1$ is  precisely (\ref{4.2new}).

Assume the Lemma is known  for
$v=w_{\beta_1}\cdots w_{\beta_{\ell-1}}$.
Recall from (\ref{Phiwparam}) that
\[
\Phi_v \ = \ \Phi_+\cap v^{-1}\Phi_{-}\ = \ \{\beta_{\ell-1}, w_{\beta_{\ell-1}}\beta_{\ell-2},\ldots, w_{\beta_{\ell-1}}\cdots w_{\beta_{2}}\beta_1\}\,.
\]
and $\Phi_{w}=w_{\beta_{\ell}}\Phi_v\cup \{\beta_{\ell}\}$ (cf.~(\ref{phivandw})).
  By induction we have
\begin{multline}\label{Mellminus1}
   \sum_{m_1,\ldots, m_{\ell}\in\mathbb{Z}}a(w_{\beta_1}u_{\beta_1}(x_1+m_1)\cdots  w_{\beta_{\ell}}u_{\beta_{\ell}}(x_{\ell}+m_{\ell})g)^{\lambda+\rho} \ \ \le \qquad\qquad\qquad\qquad\qquad\qquad\qquad\qquad\qquad \\
 \   M^{\ell-1}\sum_{m_{\ell}\in\mathbb{Z}}
a(w_{\beta_{\ell}}u_{\beta_{\ell}}(x_{\ell}+m_{\ell})g)^{v^{-1}(\lambda+\rho)}\prod_{\alpha\in\Phi_v}\left(1+a(w_{\beta_{\ell}}u_{\beta_{\ell}}(x_{\ell}+m_{\ell})g)^\alpha\right)\\
 = \ M^{\ell-1}\sum_{S\subset \Phi_v} \sum_{m_{\ell}\in\mathbb{Z}}
a(w_{\beta_{\ell}}u_{\beta_{\ell}}(x_{\ell}+m_{\ell})g)^{v^{-1}(\lambda+\rho)+\sum_{\alpha\in S}\alpha}.
\end{multline}
The $i=\ell-1$ case of
Lemma \ref{lem-aqa} applied to $\lambda-\epsilon\rho$ gives that
\[1 \
< \
\langle v^{-1}(\lambda-\epsilon\rho+\rho)+\sum_{\alpha\in S}\alpha, \beta_{\ell}^\vee\rangle   \ = \
\langle v^{-1}(\lambda+\rho)+\sum_{\alpha\in S}\alpha, \beta_{\ell}^\vee\rangle
-\epsilon\langle \rho ,v\beta_\ell^\vee\rangle
\]
for any $S\subset\Phi_v$.  As $v\beta_\ell$ is a positive root by (\ref{Phiwparam}), $\langle \rho ,v\beta_\ell^\vee\rangle$ at least 1 and hence $\langle v^{-1}(\lambda+\rho)+\sum_{\alpha\in S}\alpha, \beta_{\ell}^\vee\rangle$  must be at least $1+\epsilon$.
This shows that the assumptions of  Lemma \ref{4.1} (with $\alpha=\beta_\ell$) apply to the $m_\ell$-sum, therefore showing
(\ref{Mellminus1}) is bounded by
\[
M^{\ell}\sum_{S\subset\Phi_v} a^{w_{\beta_{\ell}}v^{-1}(\lambda+\rho)+w_{\beta_{\ell}}\sum_{\alpha\in S}\alpha}(1+a^{\beta_{\ell}})
 \ = \  M^{\ell}a^{w^{-1}(\lambda+\rho)}\prod_{\alpha\in \Phi_{w}}(1+a^\alpha).
\]
\end{proof}

\begin{proof}[\bf Proof of Theorem~\ref{thm-intro}:]

We must show that  the series $E_\lambda(g)$ converges absolutely for $g\in \Gamma UA_\mathfrak{C}K$, whenever  $\lambda\in\mathfrak{h}^*_\mathbb{C}$ satisfies $\operatorname{Re}(\lambda)-\rho\in \mathcal{C}^*$ and   Property \ref{conj} holds.  It  is equivalent to show the same assertion for $g\in UA_\mathfrak{C}K$ and $\lambda\in \frak h^*$, since the sums defining $E_\lambda(\gamma g)$ and $E_\lambda(g)$ contain the same terms and are term-by-term bounded by $E_{\operatorname{Re}(\lambda)}(g)$.

Let  $\Gamma_w=\Gamma\cap BwB=\Gamma\cap BwU_w$. Then $\Gamma_w$ is left-invariant under $\Gamma\cap  B$ and right-invariant under $\Gamma\cap U_w$.
Group the terms in the definition of Eisenstein series
\[
E_\lambda(g)=\sum_{\gamma\in (\Gamma\cap B)\backslash \Gamma}\Phi_\lambda(\gamma g)=\sum_{\gamma\in (\Gamma\cap B)\backslash \Gamma}a(\gamma g)^{\lambda+\rho}
\]
as
\begin{equation}\label{Elambda2sum}
  E_\lambda(g)=\sum_{w\in W}\sum_{\gamma_1\in (\Gamma\cap B)\backslash \Gamma_w/( \Gamma\cap U_w)}\sum_{\gamma_2\in\Gamma\cap U_w}a(\gamma_1\gamma_2g)^{\lambda+\rho}.
\end{equation}
Write $\gamma_1=utwu'$, where $u\in U$, $t\in A$, $u'\in U_w$, so that
\begin{equation}\label{ag1g2g}
a(\gamma_1\gamma_2 g)= a(utwu'\gamma_2 g)=t\cdot a(wu'\gamma_2g).
\end{equation}
Writing $w=w_{\beta_1}\cdots w_{\beta_\ell}$ as  an admissible  word (Definition \ref{def-adm}) and using Lemma \ref{lem-aqq}, we can express $wu'\gamma_2$ in the form
\[
w_{\beta_1}u_{\beta_1}(x_1+m_1)\cdots w_{\beta_\ell}u_{\beta_\ell}(x_\ell+m_\ell)
\]
for some $x_1,\ldots, x_\ell\in\mathbb{R}$ (which parameterize $u'$) and $m_1,\ldots, m_\ell\in\mathbb{Z}$ (which parameterize the sum over $\gamma_2$).

 Lemma \ref{4.5} bounds $\sum\limits_{\gamma_2\in\Gamma\cap U_w} a(\gamma_1\gamma_2g)^{\lambda+\rho}$ by
\begin{equation}\label{missingstar}
t^{\lambda+\rho}M^{\ell(w)}
a^{w^{-1}(\lambda+\rho)}\prod_{\alpha \in \Phi_w}(1+a^\alpha),
\end{equation}
where $M$ is the constant in Lemma~\ref{4.5}.
In particular,   the sum over $\gamma_2$ converges absolutely.
Since $\frak C$ is contained in the Tits cone  $\mathop{\cup}_{w \in W}  w\overline{  \mathcal C}$, we may  write  $a=w_0   a_0w_0^{-1}$ for some $w_0\in W$ and $a_0$ such that $a_0^\alpha\ge 1$ for any $\alpha>0$. Since $w_0^{-1}\alpha>0$ for all but finitely many $\alpha>0$, we see that $a^\alpha=a_0^{w_0^{-1}\alpha}\ge 1$ for all but finitely many $\alpha>0$. Hence there exists a constant $C_a>0$ which depends continuously on $a$ such that
\begin{equation} \label{prod-estimate}
\prod_{\alpha \in \Phi_w}(1+a^\alpha) \ \leq  \ C_a 2^{\ell(w)}\prod_{\alpha \in \Phi_w}a^\alpha \ = \ C_a 2^{\ell(w)}a^{\rho-w^{-1}\rho},
\end{equation}
since the product has $2^{\ell(w)}$ terms and $a^\alpha$ is bounded below by a constant depending only on $a$ (cf.~(\ref{rhominuswrho})).

The only remaining manifestation of $\gamma_1$ (aside
from $w$) is $t=t(\gamma_1)$. Recall the notation $\Phi_\lambda(g)=a(g)^{\lambda+\rho}$ and the Gindikin--Karpelevich integral (see (\ref{eee}))
\begin{multline}
\int_{(\Gamma \cap U_w) \backslash U_w} \sum_{\gamma_1 \in (\Gamma \cap B) \backslash \Gamma_w /(\Gamma \cap U_w)} \ \sum_{\gamma_2 \in \Gamma \cap U_w} \Phi_\lambda (\gamma_1 \gamma_2 n) \,dn
   \\ = \ \prod_{\alpha \in \Phi_{w^{-1}}} \frac {\xi ( \la \lambda, \alpha^\vee \ra )}{\xi ( 1+ \la \lambda, \alpha^\vee \ra )}\, .
\end{multline}
 We recall from (\ref{ag1g2g}) that $\Phi_\lambda(\gamma_1 \gamma_2n)= \Phi_\lambda (utwu'\gamma_2n) = t^{\lambda+\rho}\Phi_\lambda(wu'\gamma_2n)$, with $t=t(\gamma_1)$, and formally obtain
\begin{align*}
 \int_{(\Gamma \cap U_w) \backslash U_w} &\sum_{\gamma_1 \in (\Gamma \cap B) \backslash \Gamma_w /(\Gamma \cap U_w)} \ \sum_{\gamma_2 \in \Gamma \cap U_w} \Phi_\lambda (\gamma_1 \gamma_2 n) \,dn
 \\ & = \ \sum_{\gamma_1 \in (\Gamma \cap B) \backslash \Gamma_w / (\Gamma \cap U_w)} \int_{U_w} t(\gamma_1)^{\lambda+\rho} \Phi_\lambda (wu'n')\,dn'  \\ & = \  \sum_{\gamma_1 \in (\Gamma \cap B) \backslash \Gamma_w / (\Gamma \cap U_w)} t(\gamma_1)^{\lambda+\rho} \int_{U_w} \Phi_\lambda (wn')\,dn'\,.
\end{align*}
The last integral is the absolutely-convergent Gindikin-Karpelevich integral
\[ \int_{U_w} \Phi_\lambda(w n') dn' \ = \  \prod_{\alpha >0, w^{-1} \alpha <0} \frac {\Gamma_\R (\la \lambda, \alpha^\vee \ra )}{\Gamma_\R (1+\la \lambda, \alpha^\vee \ra )}  \]
(cf.~(\ref{eee})), and so the unfolding step above is justified by the Fubini Theorem.   We thus obtain
\[
\sum_{\gamma_1\in (\Gamma\cap B)\backslash \Gamma_w/(\Gamma\cap U_w)} t(\gamma_1)^{\lambda+\rho} \ = \
\prod_{\alpha>0, w^{-1}\alpha<0}\frac{\zeta(\langle \lambda,\alpha^\vee\rangle)}{\zeta(1+\langle \lambda,\alpha^\vee\rangle)},
\]
which is bounded for $w\in W$, say, by a constant $C_\lambda$.
Combining the above results with \eqref{missingstar} and \eqref{prod-estimate}, we find that
\[
E_\lambda(g) \ \leq \  C_\lambda \,C_a\sum_{w\in W}(2M)^{\ell(w)}a^{w^{-1}\lambda+\rho}\,,
\]
which converges absolutely by Theorem \ref{mainthm}.
\end{proof}

We end this paper with the observation that in our situation, $\Gamma UA_{\mathfrak C}K$ is strictly larger than $UA_{\mathfrak C}K$  (see the comments in the second paragraph following Theorem \ref{thm-intro}).

\begin{proposition} \label{nonstable}
Let $\mathsf A$ be a generalized Cartan matrix of indefinite type as in \cite[Theorem 4.3(Ind)]{K}, and further assume that $\mathsf A$ is non-singular, as before.  Let $G_{\mathbb R}$ be the complete Kac–Moody group associated with $\mathsf A$ as in Section \ref{Zform}. Then $U A_{\mathfrak C} K$ is not invariant under multiplication by $\Gamma$ on the left.
\end{proposition}

\begin{proof}
Recall that $\frak{C}$  is the interior of the Tits cone $X:=\bigcup_{w\in W}w\overline{\mathcal{C}}$.
Since $G$ is of indefinite type,  by \cite[Proposition 5.8c)]{K} the closure of $X$ is given by
\[
\overline{X} = \{h\in \frak{h}\,\colon\, \langle \alpha, h\rangle \geq 0 \textrm{ for all }\alpha\in \Phi_+^{\rm im}\},
\]
where $\Phi_+^{\rm im}$ is the set of positive imaginary roots.
 By \cite[Theorem 5.6c)]{K}, there exists some $\alpha\in \Phi_+^{\rm im}$ such that $\langle \alpha, \alpha_i^\vee \rangle<0$ for all $i\in I$. It follows that $\alpha_i^\vee\not\in \overline{X}$ for any $i\in I$.

Fix a simple coroot $\alpha_i^\vee \notin\overline{X}$.
Since $0\in \overline{X}  = \overline{\frak{C}}$, $A_{\frak{C}}$ contains elements $h$ arbitrarily close to the identity. Recall that $u_{-\alpha_i}(1)\in\Gamma$. For  $h\in A_{\frak{C}}$, we have that
\[
u_{-\alpha_i}(1)h = h u_{-\alpha_i}(t),\quad  \textrm{where}\quad  t=h^{\alpha_i}.
\]
 The ${\rm SL}(2, \mathbb{R})$ calculation
\[
\begin{pmatrix} 1 & 0 \\ t & 1\end{pmatrix} = \begin{pmatrix} \frac{1}{\delta(t)} & \frac{t}{\delta(t)} \\ 0 & \delta(t)\end{pmatrix} \begin{pmatrix} \frac{1}{\delta(t)} & -\frac{t}{\delta(t)} \\ \frac{t}{\delta(t)} & \frac{1}{\delta(t)}\end{pmatrix}, \ \ \delta(t)=\sqrt{t^2+1}\,,
\]
shows that the Iwasawa $A^+$-component of $u_{-\alpha_i}(1)h\in \Gamma A_{\frak C}$ is equal to $h\cdot h_{\alpha_i}(1/\delta(t))$.  This cannot lie in $A_{\frak{C}}$ for $h$ near the identity, since $\delta(t)\rightarrow \sqrt{2}$ as $t\rightarrow 1$ and $\alpha_i^\vee\not\in\overline{X}$.
\end{proof}


\begin{thebibliography}{CERP}




\bibitem[BK]{BK} A.~Braverman and D.~Kazhdan, {\it Representations of affine Kac--Moody groups over local and global fields: a suvery of
some recent results,} European Congress of Mathematics, 91--117, Eur. Math. Soc., Z\"urich, 2013.

\bibitem[Bo]{Bo} A.~Borel, {\em Automorphic forms on reductive groups.} Automorphic forms and applications, 7--39,
IAS/Park City Math. Ser., 12, Amer. Math. Soc., Providence, RI, 2007.


\bibitem[BFH]{BFH} D.~Bump, S.~Friedberg and J.~Hoffstein, {\em On some applications of automorphic forms to number theory},  Bull. Amer. Math. Soc. (N.S.) {\bf 33} (1996), no. 2, 157--175.



%

\bibitem[CG]{CG}   L.~Carbone and H.~Garland,
 {\it Existence of Lattices in  Kac--Moody Groups over Finite Fields},
 Communications in Contemporary
Math, Vol 5, No.5,  (2003)
 813--867.


\bibitem[CLL]{CLL}  L.~Carbone, K.-H.~Lee and D.~Liu, {\em Eisenstein series on rank $2$ hyperbolic Kac--Moody groups}, Math. Ann. {\bf 367} (2017), 1173--1197.

\bibitem[DHH+]{DHHKN}  T.~Damour, A.~Hanany, M.~Henneaux, A.~Kleinschmidt   and H.~Nicolai,  {\it Curvature corrections and Kac--Moody compatibility conditions}, Gen. Rel. Grav. 38:1507--1528 (2006).


\bibitem[DKN]{DKN} T.~Damour, A.~Kleinschmidt and H.~Nicolai, {\it Constraints and the $E_{10}$ coset model},  Class. Quant. Grav. 24:6097--6120, (2007).

\bibitem[DN]{DN} T.~Damour and H.~Nicolai,  {\it Higher order M--theory corrections and the Kac--Moody algebra $E_{10}$}, Class. Quant. Grav. {\bf 22} (2005) 2849--2880.

\bibitem[DGH]{DGH} T.~De Medts, R.~Gramlich  and M.~Horn, {\it Iwasawa decompositions of split Kac--Moody groups}, J. Lie Theory {\bf 19} (2009), no. 2, 311--337.

%


\bibitem[Fl]{Fl} P.~Fleig,   {\it Kac--Moody Eisenstein series in string theory}, Dissertation, University of Berlin, (2013).


\bibitem[FK]{FK} P.~Fleig and A.~Kleinschmidt, {\it Eisenstein series for infinite-dimensional $U$-duality groups}, J. High Energ. Phys. (2012) 2012: 54.

\bibitem[FKP]{FKP}  P.~Fleig, A.~Kleinschmidt and D.~Persson, {\it Fourier expansions of Kac--Moody Eisenstein series and degenerate Whittaker vectors}, Commun. Num. Theor. Phys. {\bf 08} (2014), 41--100.

\bibitem[FGKP]{FGKP} P.~Fleig, H.~Gustafsson,  A.~Kleinschmidt and D.~Persson, {\it Eisenstein Series and Automorphic Representations
With Applications in String Theory}, Cambridge Studies in Advanced Mathematics {\bf 176}, Cambridge University Press, 2018.

 \bibitem[Ga]{Ga}  O.J.~Ganor, {\it Two conjectures on gauge theories, gravity and infinite dimensional Kac--Moody groups}, arXiv:hep-th/9903110

%
%

\bibitem[G04]{G04} H.~Garland,  {\it Certain Eisenstein series on loop groups: convergence and the constant term},  Algebraic Groups and Arithmetic, Tata Inst. Fund. Res., Mumbai, (2004), 275--319.

\bibitem[G06]{G06} \bysame,  {\it Absolute convergence of Eisenstein series on loop groups}, Duke Math. J. {\bf 135} (2006), no. 2, 203--260.

\bibitem[GMS1]{GMS1} \bysame, {\em Eisenstein series on loop groups: Maass-Selberg relations. I}, Algebraic groups and homogeneous spaces, 275--300, Tata Inst. Fund. Res. Stud. Math., Tata Inst. Fund. Res., Mumbai, 2007.

\bibitem[GMS2]{GMS2} \bysame,  {\it Eisenstein series on loop groups: Maass-Selberg relations. II,} Amer. J. Math. {\bf 129} (2007), no. 3, 723--784.

\bibitem[GMS3]{GMS3} \bysame, {\it Eisenstein series on loop groups: Maass-Selberg relations. III,} Amer. J. Math. {\bf 129} (2007), no. 5, 1277--1353.

\bibitem[GMS4]{GMS4} \bysame, {\it Eisenstein series on loop groups: Maass-Selberg relations. IV}, Lie algebras, vertex operator algebras and their applications, 115--158, Contemp. Math. {\bf 442}, Amer. Math. Soc., Providence, RI, 2007.

\bibitem[G11]{G11} \bysame,  {\it On extending the Langlands-Shahidi method to arithmetic quotients of loop groups}, Representation theory and mathematical physics, 151--167, Contemp. Math. {\bf 557}, Amer. Math. Soc., Providence, RI, 2011.

\bibitem[GMP]{GMP} H.~Garland, S.D.~Miller  and M.M.~Patnaik, {\it Entirety of cuspidal Eisenstein series on loop groups},  Amer. Jour. of Math. {\bf 139} (2017), 461--512.



\bibitem[GMRV]{GMRV} M.~Green, S.D.~Miller, J.~Russo and P.~Vanhove, {\em Eisenstein series for higher-rank
groups and string theory amplitudes}, Commun. Number Theory Phys. {\bf 4} (2010), no. 3, 551--596.


\bibitem[GMV]{GMV} M.~Green, S.D.~Miller and P.~Vanhove, {\em
Small representations, string instantons, and Fourier modes of Eisenstein series}, with an appendix by D. Ciubotaru and P. E. Trapa, J. Number Theor. {\bf 146} (2015), 187--309.



\bibitem[GRV]{GRV}  M.B.~Green, J.G.~Russo  and P.~Vanhove,  {\it Automorphic properties of low energy string amplitudes in various dimensions}, Phys.\ Rev.\ D {\bf 81} (2010) 086008, arXiv:1001.2535.


\bibitem[K]{K} V.G.~ Kac, {\it Infinite dimensional Lie algebras},  Cambridge University Press, 1990.

\bibitem[KP]{KP}  V.~Kac and D.~Peterson, {\it  Defining relations of certain infinite-dimensional groups}, Ast\'erisque Num\'ero Hors S\'erie (1985), 165--208.

%

\bibitem[Ka]{Ka} M.~Kapranov, {\it The elliptic curve in the S-duality theory and Eisenstein series for Kac--Moody groups}, math.AG/0001005.

\bibitem[Kim]{Kim} H.H.~Kim, \newblock
{\em Functoriality for the exterior square of
${\rm GL}\sb 4$ and the symmetric fourth of
${\rm GL}\sb 2$}, with appendix 1 by D.~Ramakrishnan
and appendix 2 by Kim and P.~Sarnak,
\newblock J. Amer. Math. Soc. {\bf 16} (2003), no. 1, 139--183.



\bibitem[KimSh]{KimSh} H.H.~Kim and F.~Shahidi,
\newblock {\em Functorial products for
${\rm GL}\sb 2\times{\rm GL}\sb 3$ and the symmetric
cube for ${\rm GL}\sb 2$}, with an appendix by C.J.~Bushnell and G.~Henniart,
\newblock Ann. of Math. (2) {\bf 155} (2002), no. 3, 837--893.

\bibitem[Ku]{Kumar} S.~Kumar, {\it Kac--Moody groups, their flag varieties and representation theory}, Progress in Mathematics {\bf 204}, Birkh{\"a}user Boston, Inc., Boston, MA, 2002.

\bibitem[La1]{La1}
R.P.~Langlands, {\em Euler products}, Yale Mathematical Monographs
{\bf 1}, Yale University Press, New Haven, Conn.-London, 1971.

\bibitem[La2]{LA1} \bysame, {\em On the functional equations satisfied by Eisenstein series},
Lecture Notes in Mathematics, Vol. {\bf 544}, Springer-Verlag, Berlin-New York, 1976.

\bibitem[LL]{LL}   K.-H.~Lee and P.~Lombardo, {\em Eisenstein series on affine Kac--Moody groups over function fields}, Trans. Amer. Math. Soc. {\bf 366} (2014), no. 4, 2121--2165.
%

\bibitem[Li]{Li} D.~Liu, {\it Eisenstein series on loop groups}, Tran. Amer. Math. Soc. {\bf 367} (2015), no. 3, 2079--2135.

\bibitem[Lo]{Lo} E.~Looijenga, {\em Invariant theory for generalized root systems}, Inv. Math. {\bf 61} (1980), 1--32.


\bibitem[MW]{MW}
C.~Moeglin, J.-L.~Waldspurger,
{\it Spectral decomposition and Eisenstein series,} Cambridge Tracts in Mathematics, 113. Cambridge University Press, Cambridge, 1995.
\bibitem[P]{P} M.M.~Patnaik, {\it Geometry of loop Eisenstein series}, Ph.D. thesis, Yale University, 2008.



\bibitem[Sh]{Sh} F.~Shahidi, {\em Infinite dimensional groups and automorphic $L$-functions}, Pure Appl. Math. Q. {\bf 1} (2005), no. 3, part 2, 683--699.


\bibitem[T]{T} J.~Tits, {\it Uniqueness and presentation of Kac-Moody groups over fields}, Jour. of Alg. {\bf 105} (1987), 542--573.


\bibitem[W]{W1} P.C.~West, {\it E(11) and M theory}, Class.\ Quant.\ Grav.\ {\bf 18} (2001) 4443, hep-th/0104081.



\end{thebibliography}
\end{document}